\documentclass[12pt]{article}

\usepackage{amsthm,amsmath,latexsym,amssymb}
\usepackage[dvips]{graphics}
\usepackage{epsfig}
\usepackage{graphicx}
\usepackage{color}
\usepackage{hyperref}

\usepackage{epsf}
\usepackage{psfrag}
\input epsf.tex
\newcommand{\be}{\begin{equation}}
\newcommand{\ee}{\end{equation}}
\newcommand{\bes}{\begin{equation*}}
\newcommand{\ees}{\end{equation*}}

\newcommand{\eps}{\varepsilon}
\renewcommand{\P}{\mathbb P}

\newcommand{\C}{\mathbb C}
\newcommand{\CC}{\overline{\mathbb C}}
\renewcommand{\H}{\mathbb H}

\newcommand{\PP}{\mathbb P}

\newcommand{\dist}{\mathrm{dist}}
\newcommand{\diam}{\mathrm{diam\,}}

\newtheorem{thm}{Theorem}[section]
\newtheorem{prop}[thm]{Proposition}
\newtheorem{remark}[thm]{Remark}
\newtheorem{defn}[thm]{Definition}
\newtheorem{cor}[thm]{Corollary}
\newtheorem{example}[thm]{Example}
\newtheorem{lemma}[thm]{Lemma}

\begin{document}

%\baselineskip=1.2\baselineskip
%\begin{doublespace}

\title{\bf On the Riemann surface type of Random Planar Maps}
\bigskip
\author{{\bf James T. Gill}\footnote{Research supported by NSF Grant DMS-1004721}~\mbox{ }{\bf and}  {\bf Steffen Rohde}\footnote{Research supported by NSF Grant DMS-0800968.}}
\maketitle

\abstract{We show that the (random) Riemann surfaces of the Angel-Schramm Uniform Infinite Planar Triangulation and of Sheffield's infinite necklace construction are both parabolic. In other words, Brownian motion on these surfaces is recurrent. We obtain this result as a corollary to a more general theorem on subsequential distributional limits of random unbiased disc triangulations, following work of Benjamini and Schramm.}

\tableofcontents
\bigskip

\section{Introduction and results}\label{intro}

Random maps (such as a random triangulation of the sphere, uniformly chosen among all triangulations with a fixed number of vertices) are of interest in probability, 
statistical physics, and  combinatorics. Significant progress on their understanding has been gained in recent years based on 
new insights in combinatorics (specifically, bijections between maps and labelled trees \cite{S}, \cite{BFG}),
in probability (the works \cite{BS}, \cite{AS},  \cite{LG}  and many others)
and in statistical physics (notably the work \cite{DS}).

The existence of certain {\it local limits} (e.g. the UIPT of Angel and Schramm \cite{AS}) 
and {\it scaling limits} (e.g. the work of J.F. Le Gall \cite{LG}) has been established, and first topological properties
have been proved (such as one-endedness of the UIPT \cite{AS}, and homeomorphicity of the scaling limits to the sphere 
\cite{LGP}, \cite{M}). There are a large number of conjectures concerning further properties of these limits,
concerning their metric and geometric structure, eg (\cite{DS}, \cite{B}, \cite{Sh}).

There are two well-developed mechanisms to give some sort of conformal structure to 
a discrete object such as a triangulation: One is by means of Koebe-Andreev-Thurston {\it circle packings}:
If the underlying graph $G$ has no loops and no double edges, then there is an
essentially unique (up to M\"obius transformation) collection
of closed discs on the sphere, with pairwise disjoint interiors, such that the tangency
graph (discs correspond to vertices of the graph, and tangencies of discs correspond to edges)
is isomorphic to $G$. See \cite{BS} and \cite{B} for this approach, and \cite{St} or \cite{R} for background on circle packings. 
A second and perhaps more natural way works for all types of triangulations and is provided by a Riemann surface structure of the simplicial complex associated with the triangulation. 
Roughly speaking, one just glues equilateral triangles isometrically along their edges, see Section 2.2 below.
In the limiting case of infinite triangulations, the complex is no longer compact, but in our setting it is 
still simply connected. By the uniformization theorem, the Riemann surface is 
thus either parabolic (conformally equivalent to the complex plane) or hyperbolic (equivalent to the disc). An equivalent characterization is that Brownian motion on the complex is either recurrent
(in the parabolic case) or transient.

In this paper, we prove parabolicity of distributional limits of ``unbiased'' rooted triangulations
(roughly, given an unrooted triangulation T, each face has the same probability of being the root; 
see Section 2.3 for details):

\begin{thm} \label{main}
Suppose that $(T,o)$ is a subsequential distributional limit of a sequence $(T_n,o_n)$ of random rooted finite unbiased disc triangulations (limit with respect to the topology induced by the combinatorial distance $d_c$ of triangulations, see Section \ref{S:2.1}). 
Suppose further that $(T,o)$ has one end almost surely, 
and that the graph distance $d_{gr}(o_n,\partial T_n)\to \infty$ in law.
Then the Riemann surface $R(T)$ is parabolic a.s.
\end{thm}

If $d_{gr}(o_n,\partial T_n)$ does not tend to $\infty$, the limiting Riemann surface has a boundary,
and reflecting Brownian motion is still recurrent. An instructive example is the 7-regular graph,
see Section 2.3 below.

There is no need to restrict to triangulations: Theorem \ref{main} and its proof remain
true with only minor changes for unbiased $m$-angulations, or slightly more generally by angulations with uniform bound on the degree of the faces.

Theorem \ref{main} is similar to the result of Benjamini and Schramm \cite{BS} concerning the recurrence of a simple random walk on 
subsequential distributional limits of unbiased graphs. In fact, our method of proof is an
adaptation of their method to the setting at hand. However, due to the use of circle packings, their
proof only works under the additional assumption of uniformly bounded vertex degress. The perhaps most significant contribution of this paper is the realization that, when working with the Riemann surface, we do not need to make this extra assumption (we only use the trivial fact that the dual graph has bounded degree).

In \cite{AS}, it was noted that the method of \cite{BS} can be adapted to show that the circle packing type of the UIPT is parabolic. It was also conjectured that simple random walk on the UIPT is recurrent a.s. As the vertex degree in the UIPT is a.s. unbounded, the method of \cite{BS} does not yield recurrence. While we are not able to prove this conjecture, 
our theorem does show that Brownian motion on the UIPT is recurrent a.s. 
In other words, we prove

\begin{cor}\label{coruipt}
The Riemann surface associated with the UIPT is a.s. parabolic. 
\end{cor}

Notice that parabolicity of the UIPT is an implicit part (though much weaker) of Conjecture 7.1 in \cite{DS}.
Another interesting random infinite surface, the ``infinite necklace", was recently constructed by Scott Sheffield, see Figure \ref{fig:rbtriangle} and Section 3.2 for a description. As another application of 
Theorem \ref{main}, we prove his
conjecture concerning parabolicity of the associated Riemann surface:

\begin{cor}\label{corneck}
Brownian motion on the infinite necklace surface is almost surely recurrent.
\end{cor}

We will now outline the proof of Theorem \ref{main}.
How can one decide if a given triangulated Riemann surface $R$ is parabolic or hyperbolic?
This {\it type problem} has been studied from many perspectives, but 
our approach is very simple and is motivated by the method of Benjamini and Schramm \cite{BS}.
Fix a conformal map $\phi:R\to\C$. 
The direct analog of their method would be to associate with each vertex $v$ of the triangulation
the {\it half-flower} $H_v\subset R$ (see Figure \ref{fig:flower}) and thus obtain a packing $P$ of the plane
by topological discs $\phi(H_v)$ with tangency graph the triangulation. The role of the circle packing theorem is now played by the uniformization theorem, namely the existence of $\phi$.
If it then were true that the discs $\phi(H_v)$ have {\it bounded geometry}
(say, inradius and outradius comparable up to constants independent of $v$), 
then parabolicity $\phi(R)=\C$ would be equivalent to showing that the set of the centers of the packing $\{\phi(v):v\in V\}$ has no finite accumulation point.
Unfortunately, due to the unboundedness of the degree, we have not been able to show that the 
half-flowers have bounded geometry. But switching to the dual graph and considering the centers of the triangles instead of the vertices, we do obtain a test for parabolicity:

\begin{prop}\label{parabolic}
Let  $\phi$ be a conformal map of $R$ into $\CC$. 
Then $R$ is parabolic if and only if the following two conditions are satisfied:
First, the set of centers of triangles $\{\phi(c)\}$ has precisely one accumulation point in 
$\CC.$ And second, $d_{gr}(v,\partial T)=\infty$ for one (hence all) vertices $v$.
\end{prop}

The proof is given in Section 2.2 and consists in showing that the {\it interstices} of the half-flower packing have bounded geometry. 

We are now in the situation of \cite{BS} and consider distributional limits $(G,o)$ of finite random 
unbiased graphs $(G_n,o_n)$. We re-interpret their Lemma 2.3 and proof of Prop 2.2 in the following
way, without the need to refer to packings. Assume that for each rooted graph
we have an embedding $g_n$ of the set of vertices into the plane that is normalized
in such a way that $g_n(o_n)=0$ and that the closest point among $\{g_n(v):v\in V_n\setminus\{o_n\}\}$
is of distance 1 from 0. Assume further that any two embeddings $g_n, g_n'$ of the {\it same} (unrooted) graph differ only by a real dilation and complex translation.
Finally, assume that $(G_n,o_n,g_n)$ converge to $(G,o,g)$ in the sense that 
$(G_n,o_n)\to (G,o)$ in distribution (with respect to combinatorial distance defined in Section 2.1) and that furthermore
$g_n\to g$ in distribution (with respect to a topology defined in Section 2.3).

\begin{prop}\label{BSLemma}
If $(G_n,o_n,g_n)$ converge to $(G,o,g)$ in the above sense, then the random set of points
$g(G)$ has at most one accumulation point in $\C$ almost surely.
\end{prop}

Notice that in Proposition \ref{BSLemma}, no assumption is made on the graphs
$(G_n,o_n,g_n)$ apart from 
distributional convergence. Particularly, it is not assumed that the graphs are planar. 
However, for our application we have to work with the dual graph of the random triangulation,
and embeddings of (centers of) faces rather than vertices. For the reason of clarity, in Section 2.3, Proposition \ref{parabolic-limit}, 
we state and prove a version of Proposition \ref{BSLemma} that is tailored to our applications.
The proofs are identical and are basically in \cite{BS}.
Theorem \ref{main} is a consequence of Proposition \ref{parabolic} and (the variant of)
Proposition \ref{BSLemma}. In order to apply Proposition \ref{BSLemma}, we need to prove
existence of the limit $(G,o,g)$. An easy way to do this is by proving compact convergence
of the conformal maps $\phi_n$ of the Riemann surfaces $R(T_n),$ which follows from Montel's theorem.
To obtain Corollary \ref{coruipt}, we approximate the UIPT by an unbiased sequence of disc triangulations, simply by removing a randomly chosen triangle. See Section 3.1 for details. Finally, to prove Corollary \ref{corneck}, we show that the boundary of the random disc
with $n$ triangles has size about $\sqrt n$ (Lemma 3.5), so that the distance of the root to the boundary 
stays bounded with probability about $1/\sqrt n$ (Lemma 3.6).

\vspace{.1in}

\noindent{\bf Acknowledgements:} The authors would like to thank Itai Benjamini and Nicolas Curien for helpful comments on a draft of this paper.

\section{Packings, triangulations, and Riemann surfaces}

In this section, we will describe the Riemann surface associated
with a triangulation via glueing of equilaterals, and provide
a simple criterion for parabolicity of such a surface.

\subsection{Triangulations}\label{S:2.1}

We adopt the terminology and definitions of the important papers \cite{BS}, \cite{AS}
on random planar triangulations. Here we restrict ourselves to very 
briefly listing and explaining  the key terms, and refer the reader to
the well-written and nicely illustrated Section 1.2 in \cite{AS} 
for motivation and details.

An {\it embedded triangulation} $T$ consists of a finite or infinite
connected graph $G$ embedded in the sphere $\CC$, together with a subset of the triangular faces (=connected components of $\C\setminus G$ whose boundary meets precisely three edges of $G$). 
We will drop the term {\it embedded}, and we will always assume that the triangulation is {\it locally finite}, that is, every point in the {\it support} $S(T)$
(= the union of $G$ and all the triangles in $T$) has a neighborhood in $\CC$ that intersects only a finite number of elements of $T.$
Two triangulations $T$ and $T'$ will be considered {\it equivalent} if there is
an orientation preserving homeomorphism of their supports corresponding $T$ and $T'$. See Definitions 1.1 and 1.5 in \cite{AS}.
There are different types of triangulations, according to what types of graphs we allow. We will only consider {\it type II triangulations} where $G$ has no loops but possibly multiple edges, and {\it type III triangulations} where $G$ has neither loops nor multiple edges (Definitions 1.2 and 1.3 in \cite{AS}).
A {\it rooted triangulation} $(T,o)$ is a triangulation $T$ together with an
oriented triangular face $o=(x,y,z)$ of $T,$ called the {\it root}. The vertex $x$ is the 
{\it root vertex} and $(x,y)$ is the {\it root edge}.

We call $T$ a {\it disc trianguation} if the support is simply connected. For finite triangulations, this implies that the complement
$\CC\setminus S(T)$ is an $m-$gon for some $m,$ so that $T$ is a triangulation of an 
$m-$gon. A vertex $v$ is an {\it interior} resp. {\it boudary vertex} if $v$ 
belongs to the interior resp. boundary of $S(T).$ Denote $\partial T$ the set of 
boundary vertices (which is empty if $S(T)$ is open).

Denote $d_{gr}(v_1,v_2)$ the {\it graph distance} on $T$ and
write $d_{gr}(v,\partial T)$ for the minimal graph distance between $v$ and
any point $v'\in \partial S(T)$ (if there is no such $v'$, we set 
$d_{gr}(v,\partial T)=\infty$).
Define the {\it combinatorial distance} between two rooted triangulations
$(T,o)$ and $(T',o')$ as 
$$d_c((T,o),(T',o')) = 1/(k+1)$$
where $k$ is the largest radius $r$ for which the combinatorial balls $B_r$ and
$B'_r$ around the root vertices $v,v'$ are equivalent, where equivalency now also requires that the root is preserved. 
(The inductive definition of $B_r$, given in \cite{AS}, Definition 4.3, is that $B_0$
is the root vertex and $B_{r+1}$ consists of all triangles incident to a vertex of $B_r$,
together with all vertices and edges of these triangles).

We will be mostly dealing with disc triangulations and their limits. 
In general, it is not true that limits of disc triangulations are disc triangulations. 

\begin{lemma} If $(T_n,o_n)$ is a sequence of disc triangulations converging to a triangulation $(T,o)$ in $d_c,$ if $d_{gr}(o_n,\partial T_n)\to\infty$, and if $T$ has one end, then $T$ is a disc triangulation.
\end{lemma}

We leave the proof of the Lemma as an exercise to the reader. None of the assumptions of the Lemma
can be omitted: To see this, consider appropriate sequences of triangulations of the cylinders $S^1\times [0,n]$ and $S^1\times[-n,n]$,
for instance by restricting the planar triangular lattice to a strip.
Both are disc triangulations converging to  triangulations of  punctured discs (namely a half-infinite and a bi-infinite cylinder).  
The first limit has one end but $d_{gr}(o_n,\partial T_n)\not\to\infty$, the second limit has two ends 
(though $d_{gr}(o_n,\partial T_n)\to\infty$).

\subsection{Riemann surfaces}\label{S:2.2}

With each triangulation we  associate a metric space $(M(T),d_M)$, by regarding each triangle of $T$ as an equilateral triangle 
and defining the distance between two points as the length of the shortest path joining them. In other words, $M(T)$ is obtained by glueing equilateral triangles according to their adjacency pattern in $T.$

This metric space can be equipped with a compatible Riemann surface structure 
by defining coordinate charts as follows. If $x$ is an interior point of a triangle, simply use that triangle as the chart, view it as a subset of the plane, and use the identity map as the projection. Similarly, if $x$ is interior to an edge, simply place the two adjacent triangles in the plane 
and again use the identity as the projection.
(If $x$ is in the boundary of the support, there will be only one triangle and we obtain a bordered Riemann surface). 
Finally, if $x$ is a vertex of degree $n,$ unless $n=6$
we cannot place the triangles in the plane so that they form a 
neighborhood of $x.$ Nevertheless, we take for the chart
% the {\it flower} $F_x$ which is 
the union of $x$ with the edges and faces adjacent to $x$. Loosely speaking, we set $x=0$ and take the map 
$z\mapsto z^{6/n}$ as the coordinate map. More precisely,
if $\Delta_1,\Delta_2,...,\Delta_n$
denote the equilaterals which meet at $x$ in cyclic order, 
we may place them in the plane so that $\Delta_j$ is the triangle 
with vertices $0, e^{2\pi i(j-1)/6}, e^{2\pi i j/6}$. Then we define
the coordinate map
for $z=r e^{it}\in \Delta_j$ by $\phi(z)=r^{6/n} e^{6 t i/n},$ using the convention
that $\frac{2\pi (j-1)}{6}< t < \frac{2\pi j}{6}$ on $\Delta_j$.
Finally, we set $\phi(0)=0$ and notice that $\phi$ is a homeomorphism 
between a neighborhood of $x$ in $M$ and a neighborhood of $0$ in $\C.$
It is obvious that coordinate changes are analytic, so that this indeed defines a
Riemann surface.

Following \cite{BS} (but using triangles instead of vertices for roots, as in \cite{AS}),
we call a random rooted triangulation $(T,o)$ {\it unbiased} if
the law $\mu$ of $(T,o)$ is in the closed convex hull of the measures $\mu_H$ which
are defined for finite triangulations $H$ as the uniform measure on rooted triangulations
$(H,o)$ where $o$ ranges over all possible 
triangles of $H.$ Informally, $(T,o)$ is unbiased if given $H$ the root $o$ is uniformly distributed
among the oriented triangles.

We now state our parabolicity criterion for Riemann surfaces of random triangulations.
It is the Riemann surface analog to Theorem 1.1 of \cite{BS}, the recurrence of 
simple random walk on distributional limits of unbiased finite planar graphs with uniformly
bounded degrees.
Notice that we allow unbounded degrees, and that we {\it assume} existence of a subsequential limit.

\newtheorem*{th1}{Theorem \ref{main}}
\begin{th1}
Suppose that $(T,o)$ is a subsequential distributional limit of a sequence $(T_n,o_n)$ of random rooted finite unbiased disc triangulations (limit with respect to the topology induced by the combinatorial distance $d_c$ of triangulations). 
Suppose further that $(T,o)$ has one end almost surely, 
and that $d_{gr}(o_n,\partial T_n)\to \infty$ in law.

Then the Riemann surface $R(T)$ is parabolic a.s.
\end{th1}
Put differently, under the assumptions of the theorem, Brownian motion on any distributional limit Riemann surface
is recurrent a.s. If we omit the assumption $d_{gr}(o_n,\partial T_n)\to \infty,$ the limiting Riemann surface could have a boundary. In this case, reflecting Brownian motion would still be recurrent.
The crucial condition on unbiasedness can be somewhat weakened. For instance, it would be enough to assume that,
given a finite triangulation $T$, for ``most" possible roots $o$ the probility that $o_n=o$ is comparable
to the reciprocal of the number of faces of $T.$ Indeed, our proof will go through without modifications if unbiasedness is replaced by the existence of a constant $C$ such that for each $\eps>0$ there is $n_0$ such that for $n\geq n_0$,
if $U_n(T)$ denotes the set of those directed triangles $o$ of $T$ for which
$$\frac1{C|F(T)|} \leq \P[o_n=o | T_n=T] \leq \frac{C}{|F(T)|},$$
then 
$$\P[o_n\in U(T) |  T_n=T] \geq 1-\eps.$$
The recurrence proof \cite{BS} used circle packings and was based on the following result from \cite{HS}:
If a collection of discs $D_v\subset\C$ with pairwise disjoint interiors has no accumulation point in the plane,
and if its tangency graph $G$ has uniformly bounded vertex degree, then $G$ is recurrent.
We will proceed similarly, by associating with each triangulation a packing of topological discs in the plane. 
One of the differences, neccessitated by our need to work with unbounded degrees, is that we work with
a packing of the dual graph rather then the original graph (in circle packings, this would correspond to the 
collection of interstices rather than the collection of discs). 
First some definitions:  
Let $v$ be a vertex of $T$ of degree $\deg v=n$.  The {\em flower} $F_v=F_v(T)\subset R$
of $v$ is the union of the $n$ equilateral triangles 
$(v,v_i,v_{i+1})$ which meet at $v$, together with the edges. 
Similarly, the {\em half-flower} $H_v$ 
is the $n$-gon formed by the $\deg(v)$ equilateral triangles 
$(v,v_i',v_{i+1}')$, where $v_i'$ is the midpoint between $v$ and $v_i,$
see Figure \ref{fig:flower}. The {\it interstice} $I_f$ of a triangular
face $f$ with vertices $(v_1,v_2,v_3)$ 
is the equilateral sub-triangle whose vertices are
the midpoints of the edges, see Figure \ref{fig:flower2}. We denote the {\it center} of a face $c_f$.
The interstices and the half-flowers form a decomposition of the surface. The name is borrowed from the theory of circle packings (see \cite{St} or \cite{R}).

\begin{figure}[htp]
\centering
\includegraphics[scale=0.40]{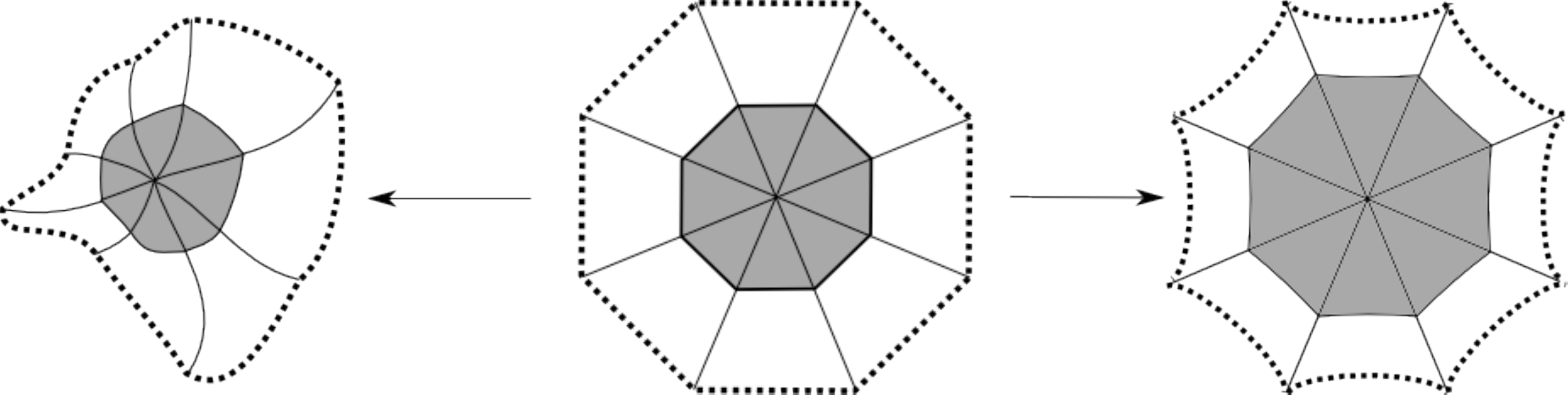}
\caption{The flower, half flower (shaded), and images under a uniformizing map and under the coordinate map}\label{fig:flower}
\end{figure}

\begin{figure}[htp] 
\centering
\includegraphics[scale=0.6]{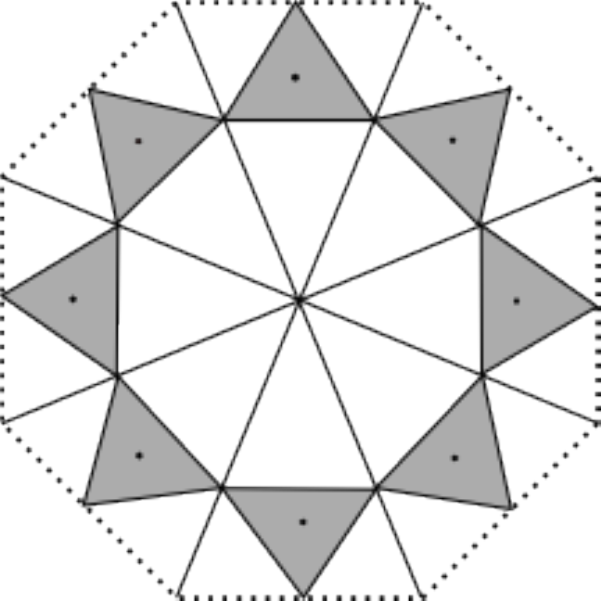}
\caption{The same flower as in \ref{fig:flower} but with the interstice triangles shaded.}\label{fig:flower2}
\end{figure}

For the remainder of this section, we assume that $T$ is a disc triangulation so that $R$, the Riemann surface associated with $T$, is simply connected. By the Koebe uniformization theorem, there is a conformal homeomorphism into the sphere $\CC.$ The following simple criterion for parabolicity is key to our approach.

\newtheorem*{prop1}{Proposition \ref{parabolic}}
\begin{prop1}
Let  $\phi$ be a conformal map of $R$ into $\CC$. 
Then $R$ is parabolic if and only if the following two conditions are satisfied:
First, the sequence of centers of interstices $\phi(c_f)_{f\in F(T)}$ has precisely one accumulation point in $\CC.$ And second, $d_{gr}(v,\partial T)=\infty$
for one (hence all) vertices $v$.
\end{prop1}

\begin{remark}\label{rmk}
{\rm
 The proposition also holds if we replace the condition of ``precisely one accumulation point in $\CC$'' with`` finitely many accumulation points in $\CC$''.  This variant with finitely many accumulation points is what we will eventually use in the proof of Theorem \ref{main}.
}
\end{remark}

\vspace{.5cm}

The condition on the centers of interstices can be replaced by the condition that the sequence of interstices $\phi(I_f)_{f\in F(T)}$ has at most one accumulation point in $\CC.$ 
We believe that the condition is also equivalent to requiring
that the half-flowers, or their centers, have only one accumulation point. 
This would be the most direct analog of the notion of ``Circle Packing parabolicity" of He and Schramm, \cite{HS}, and the aforementioned approach in \cite{BS}.
Since we do not assume the degrees of the vertices to be bounded, we have no control over the geometry of the half-flowers. Indeed, it is easy to see that the hyperbolic diameter 
of the half-flowers within the combinatorial ball $B_1$ tends to infinity as the degree of the root vertex tends to infinity.
To see what could happen when no bounded geometry assumption is made,  consider Figure \ref{fig:packingwithlotsoflimitpoints}.
\begin{figure}[htp]
\centering
\includegraphics[scale=0.40]{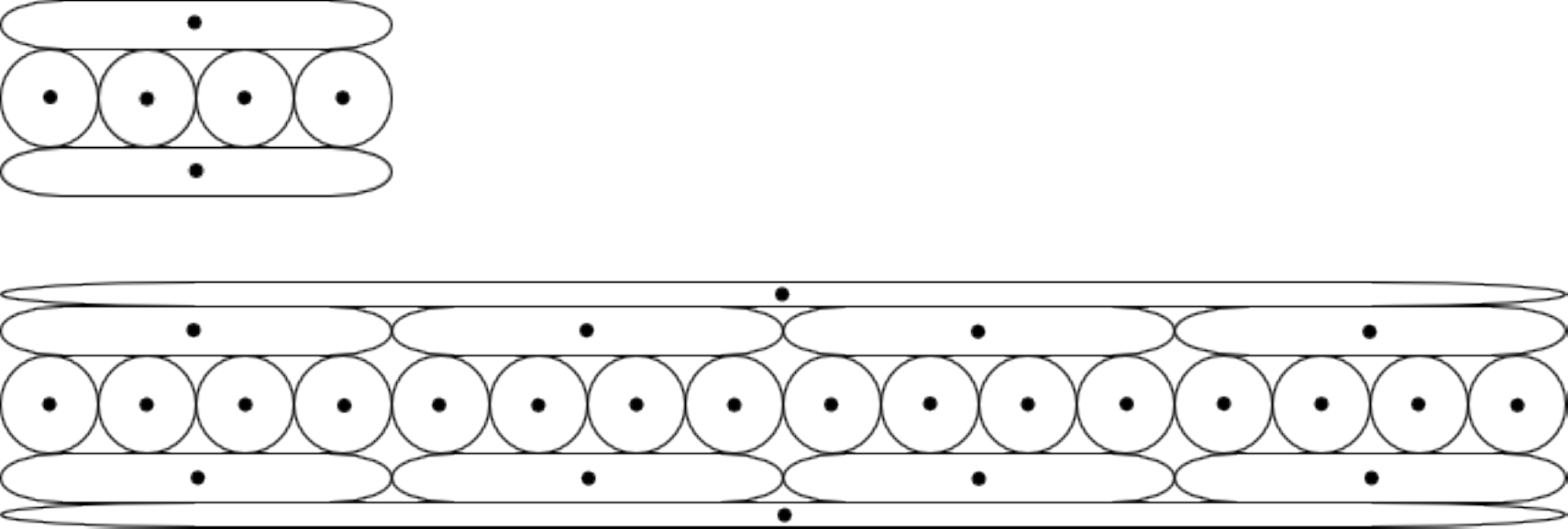}
\caption{A packing of a hyperbolic graph by sets of unbounded geometry}\label{fig:packingwithlotsoflimitpoints}
\end{figure}
It shows the first two steps of an iterative construction of a (deterministic) packing of a hyperbolic (transient) graph
(it is similar in spirit to the image of the Figure \ref{fig:lattice} under the homeomorphism $z\mapsto z/(1-|z|)$ between the unit disc and the plane).
The packing has two infinite lines of accumulation, but it is easy to associate with each of the packed sets a point ``well within" the set 
such that this collection of points has no accumulation point in the plane.  
The problem is that the points no longer properly ``represent'' the topological discs because there 
is no uniform assumption on the geometry of the discs. We get around this problem by considering the interstices.
Notice that in Figure \ref{fig:packingwithlotsoflimitpoints} the interstices do not have bounded geometry.
Our proof of the Proposition crucially relies on the fact that the Riemann surface interstices have bounded geometry, independently of the degrees of the vertices:

\begin{lemma}\label{bddgeometry} 
There is a constant $C$ such that  for all
simply connected open $R$ as above, all faces $f$ of $R$ with $d_{gr}(f,\partial R)\geq2$,
and all conformal maps $\phi:R\to\C$,
\begin{equation}\label{geometry}
B(\phi(c_f), \frac1C \diam \phi(I_f)) \subset \phi(I_f) \subset B(\phi(c_f), C \diam \phi(I_f)).
\end{equation}
If $f'$ is a face adjacent to $f,$ then
\begin{equation}\label{centergeom}
\frac1C\leq  \frac {|\phi(c_f)-\phi(c_{f'})|}{\diam \phi(I_f)} \leq C.
\end{equation}
If furthermore $R$ is hyperbolic, then the hyperbolic diameter of $I_f$ satisfies
\begin{equation}\label{hypdiam}
\diam_{hyp}(I_f)<C.
\end{equation}
\end{lemma}

\begin{proof}
Consider the union $U$ of 
$f$ with the triangles sharing an edge with $f$. Then $U$ consists of either four triangles, or (in the presence of double edges) $U$ consists of three triangles. In both cases, the interstice $I_f$ and the line segment $[c_f,c_{f'}]$ are compact subsets of $U$ 
and thus of finite hyperbolic diameter within  $U$.
If $\psi$ is a conformal map of $U$ onto the unit disc, such that $c_f$ maps to $0,$ then \eqref{geometry} and
\eqref{centergeom} are trivially satisfied since there are only two such maps $\psi.$
Now standard applications of the Koebe distortion theorem gives \eqref{geometry} and \eqref{centergeom}.
Since the hyperbolic distance decreases
as the domain increases (essentially Schwarz' Lemma), $I_f$ and $[c_f,c_{f'}]$ have uniformly
bounded hyperbolic diameter independent of the triangulation. This proves \eqref{hypdiam}.
\end{proof}

\begin{proof}[Proof of Proposition \ref{parabolic}]
Since $R$ is not conformally $\CC,$ we may assume that $\phi(R)\subset\C.$ 

If $R$ is parabolic, then $d_{gr}(v,\partial T)=\infty$, for otherwise there
would be a boundary edge at finite combinatorial distance from $v$ and Brownian motion would be transient. Also, in that case $\phi$ is a proper homeomorphism between $R$ and $\C$ so that $\infty$ is the only limit point.

To show the converse, assuming that $R$ is hyperbolic and that 
$d_{gr}(v,\partial T)=\infty$ we will show that the sequence $\phi(c_f)$ has infinitely many accumulation points.
Pick a point $w\in\partial \phi(R)$ (boundary with respect to the euclidean metric) such that there is a curve $\gamma$ in $\phi(R)$ joining $\phi(v)$ and
$w.$ Since $\phi(R)$ is hyperbolic, $\partial \phi(R)$ is non-empty, connected, and the set of such $w's$ is dense in $\partial \phi(R)$ 
(for instance, the points in  $\partial \phi(R)$ closest to some interior point will do).
The curve $\phi^{-1}(\gamma)$ eventually leaves every compact set and therefore has to meet infinitely many interstices
$I_1,I_2,...$ In particular, there is a sequence of points $z_j\in I_j$ such that 
$\phi(z_j)\to w$.
By  \eqref{hypdiam} and the Koebe distortion theorem, it follows that $\diam \phi(I_f)\leq C' |w-\phi(z_j)|$ so that
$\phi(c_j)\to w$ as well.
\end{proof}

\subsection{Center Embeddings}

In this section, we develop a  variant of the Benjamini-Schramm technique \cite{BS}, suitable for our needs.
See \cite{BC} for a different (but related) variant. One of the key ideas in \cite{BS} was to show that the centers
of the circle packings associated with a distributional limit of unbiased random triangulations of the sphere
have only one limit point almost surely. In case of uniformly bounded vertex degrees, this implied recurrence
of the simple random walk, by \cite{HS} and \cite{McC}.

With  our parabolicity criterion Proposition \ref{parabolic} in mind (and thinking of the Riemann surface interstice-packing as a substitute for the circle packings), we are interested in
finite or countably infinite sets of points $\phi(c_f)_{f\in F(T)},$ together with the
combinatorics of the triangulation. We thus make the following 
\begin{defn}
A {\bf center embedding} is a triple $E=(T,o,g)$, where $(T,o)$ is a rooted triangulation and 
$g:F(T)\to\C$ an injective map. We also require $g$ to be normalized by $g(o)=0$ and such that
the closest point is of distance 1,
\begin{equation}\label{cenorm}
\inf_{f\neq o} |g(f)| = 1.
\end{equation}
\end{defn}
We equip the space $\mathcal{E}$ of normalized center embeddings with a topology in such a way that
$$(T_n,o_n,g_n)\to (T,o,g)$$
if and only if both $(T_n,o_n)\to (T,o)$ in the graph metric, 
and $g_n(f)\to g(f)$ for each face $f\in T$. 

Writing $d=d_{gr}((T,o),(T',o')),$
this topology is easily seen to be generated by the metric
\[
d_\mathcal{E}((T,o,g),(T',o',g')) =  d +
\sum_{n = 0}^{d^{-1}} 
\frac1{2^{n+1}|B_n|}
\sum_{f \in B_n} \frac{ |g(f)-g(f')| }{1+|g(f)-g(f')|}, 
\]
where $B_n$ denotes the ball of radius $n$ in $T$, $|B_n|$ denotes the number of faces of $B_n$, and 
where $f'$ denotes the face in $T'$ corresponding to $f\in T$.
It is not hard to see that $(\mathcal{E},d_\mathcal{E})$ is separable.
Slightly abusing standard terminology, we call two functions $g$ and $g'$ {\it similar} 
if $g'= a g+b$ with real $a>0$ and complex $b$.
Thus each function $g:F(T)\to\C$ is similar to a normalized one. 

\begin{prop}\label{parabolic-limit}
Suppose that $E_n$ is a sequence of random unbiased finite center embeddings such that
the (unrooted) triangulation $T_n$ determines the embedding (in the sense that if 
$T$ and $T'$ are equivalent as planar graphs, then the embeddings $g$ and $g'$ are similar).
Then for every distributional subsequential limit $E$ of $E_n$,
the (random) set $g(F(T))$ has at most one limit point almost surely.
\end{prop}

By a random center embedding we mean a probability measure $\P$ on $\mathcal{E}$, 
and unbiasedness refers to the induced measure on rooted graphs (given a finite graph, each face
has the same probability of being the root). 
Notice that the existence of a limit is an assumption, not a conclusion.

\bigskip
\noindent
\begin{example}\label{expl}
{\rm
Instructive examples are the 6-regular triangular lattice and the 7-regular lattice
and the corresponding centers of Riemann surface interstices $E_{\infty}$ (or circle packings ${\bf P}_{\infty}$). 
If $E_n$ is obtained from the combinatorial ball $B_n$ 
by uniformly choosing a root face, then the proposition applies in both the 6- and 7- regular cases, 
and the limit $E$ (which is easily seen to exist in both cases) has no finite limit point.  In Figure \ref{fig:lattice} below, we show two versions of circle packings of the 7-regular lattice.  The first is an approximation to $E_\infty$. The 
set $E_{\infty}$ (resp. ${\bf P}_{\infty}$) clearly has infinitely many limit points in the 7-regular case.  This appears to be a contradition to our proposition.  It is not.
The explanation is that $E_{\infty}\neq E$ in the 7-regular case.  The drawing on the right of Figure \ref{fig:lattice} is a better picture for the limit embedding $E$.   It is simply what the  figure on the left becomes when one normalizes with respect to a boundary circle. Moreover, in this example, Theorem \ref{main} does not apply: Because of the linear isoperimetric inequality, the distance of the root to the boundary does not tend to infinity a.s.
}
\end{example}

\begin{figure}[htp]
\centering
\includegraphics[scale=0.30]{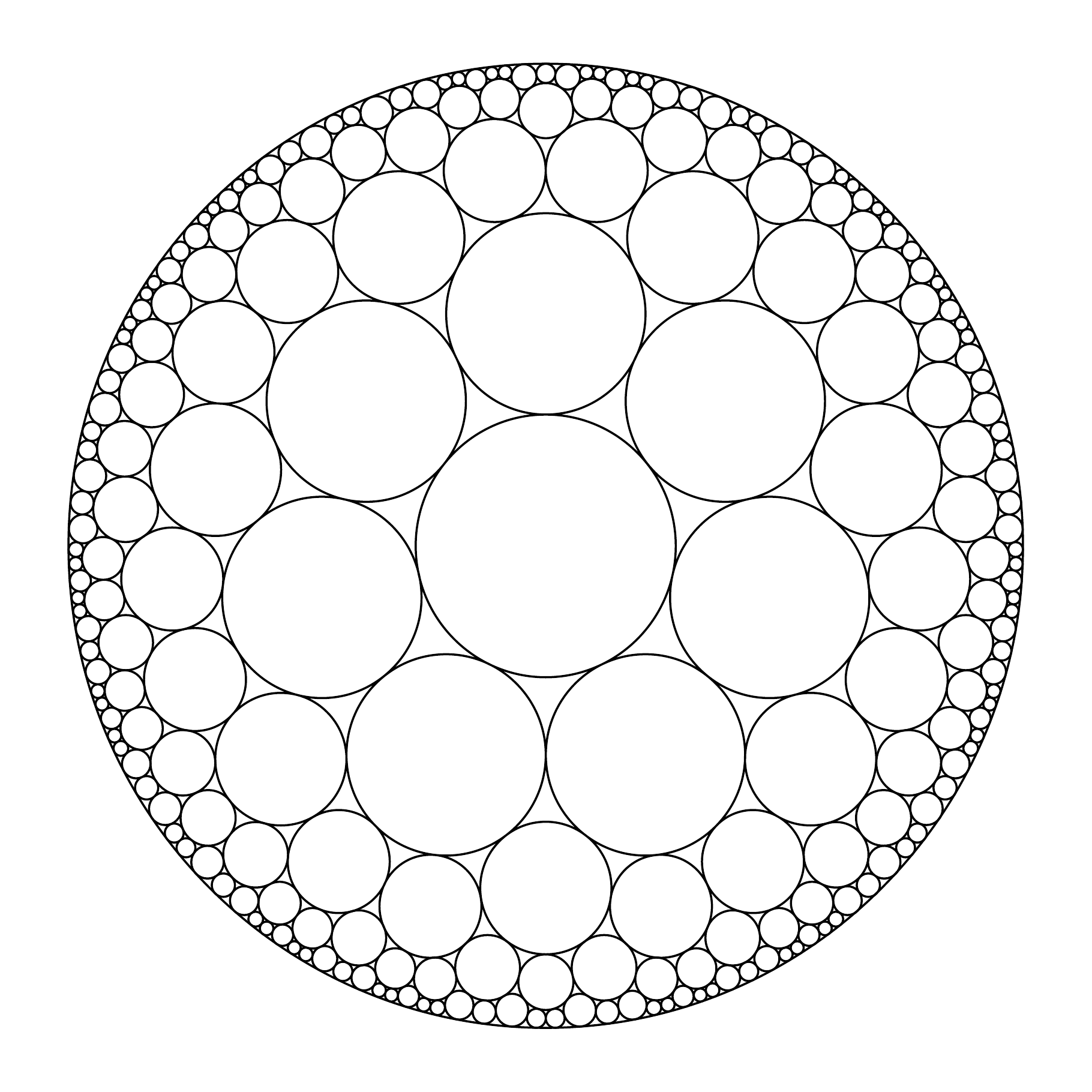} \hspace{1cm} \includegraphics[scale=0.30]{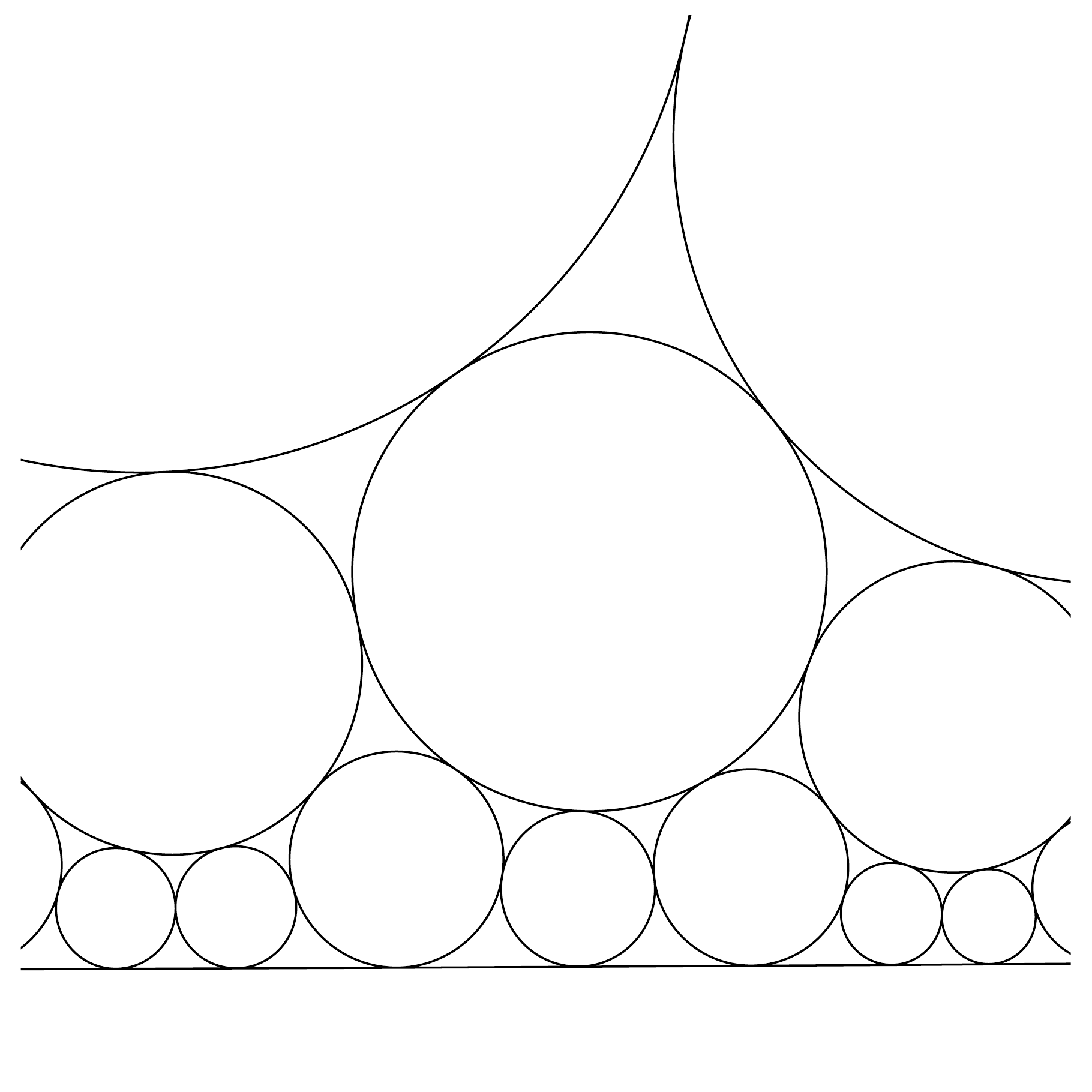}
\caption{Approximations of 7-regular lattice circle packing, two points of view}\label{fig:lattice}
\end{figure}

As mentioned in the introduction,
the proof of Proposition \ref{parabolic-limit} is essentially the same as the proof of Proposition 2.2
in \cite{BS}. Since our setup is different, we go through the details for the convenience of the reader. 
The key technical result is the following magical Lemma 2.3 from \cite{BS} (also see \cite{BC} for a nice exposition and generalizations to higher dimensions). It can be viewed as a quantitative statement to the 
effect that every finite planar set ``looks like" it has at most one accumulation point, when viewed from ``most" of the points. Here is the precise statement. Let $V \subset \C$ be a finite set of points (in our application, $V$ will consist of the marked points in the packings).  For $v \in V$, the isolation radius is defined as
\[ \rho_v := \inf\{ |v-w|: w \in V \}. \]
Given  $\delta > 0$, $s>0$, and  $v \in V$, say that $v$ is $(\delta, s)$-{\em supported} if
\[ \inf_{p \in \mathbb{C}} \left| V \cap \left(D(v,\rho_v / \delta) \setminus D(p, \rho_v \delta) \right) \right| \geq s.\]

\vspace{2mm}\noindent{\bf Lemma A} (Lemma 2.3 in \cite{BS}) {\em
For every $\delta \in (0,1)$ there is a constant $c = c(\delta)$ such that for every finite set $V \subset \mathbb{C}$ and every $s \geq 2$ the proportion of $(\delta, s)$-supported points in $V$ is less than $c/s$.}
\vspace{2mm}

\begin{proof}[Proof of Proposition \ref{parabolic-limit}]
The proof is by contradiction.  Suppose that a distributional limit $E =(T,o,g)$ has two or more accumulation points with positive probability.  
Then there is $\delta>0$ and two points  $p_1,p_2 \in D(0,1/\delta)$,  $|p_1-p_2| \geq 3\delta$,
such that the set of embeddings with accumulation points in $D(p_1,\delta)$ and in $D(p_2,\delta)$  has positive probability $\alpha>0$.  
For $k>0,$ denote $\mathcal{O}_k$  the set of all embeddings such that 
 $B(p_1,\delta) $ and $B(p_2,\delta)$  contain at least  $s$ points from the embedding  corresponding to faces in  $B_{k}.$
Then there is $k=k_s$ such that the event $T\in \mathcal{O}_{k_s}$ has  probability at least $\alpha/2.$
Since $\mathcal{O}_{k_s}$ is open with respect to the metric $d_\mathcal{E}$, the
distributional convergence implies that 
the liminf of the probability of $T_n\in \mathcal{O}_{k_s}$  is greater than or equal to  $\alpha/2.$
But  the magical Lemma A, with $\delta$ fixed,  implies that $T_n\in \mathcal{O}_{k_s}$
has probability at most $c(\delta)/s$, a contradiction when $s>\alpha/(2c(\delta)).$

\end{proof}

\bigskip
\noindent
{\bf Remark:} For the purpose of proving parabolicity of the UIPT and the necklace, we could work 
in the perhaps more intuitive setting of packings of topological discs, instead of center packings.
That is, we could consider finite or infinite
collections ${\bf P} =\{P_v : v \in V\}$ of closed topological discs $P_v\subset\C$ with pairwise disjoint
interiors, together with a collection of interior points $x_v \in P_v$. 

Then the appropriate topology is induced by the metric
\[
d_\mathcal{P}({\bf P},{\bf P}') = d_c(P,P') +
\sum_{n = 0}^{d_c(P,P')^{-1}} 
\frac1{2^{n+1}|B_n|}
\sum_{v \in B_n} \Bigl( \frac{ |x_{v}-x_{v'}| }{1+|x_{v}-x_{v'}|}  +  
                  \frac{ d_H(P_{v},P_{v'}) }{1+d_H(P_{v},P_{v'})}\Bigr), 
\]
where $d_H$ is Hausdorff distance. The laws of the interstice-packings converge even in this finer topology. But Proposition \ref{parabolic-limit} is more general and the assumptions are easier to verify.

\subsection{Proof of Theorem \ref{main}}

\begin{proof}[Proof of Theorem \ref{main}]

With an eye toward Proposition  \ref{parabolic-limit}, we first associate a
center embedding with a rooted triangulation.
For each finite (unrooted) disc triangulation $T$, fix a conformal map 
$$\phi_T: R(T)\to \C.$$ Given a rooted finite disc triangulation $(T_n,o_n)$,
define a normalized conformal map by
\begin{equation}\label{phiT}
\phi_n = a \phi_{T_n} + b
\end{equation}
where $a>0$ and $b\in\C$ are chosen so that 
$$\phi_n(c_{o_n}) = 0$$
and 
\begin{equation}\label{norm}
\inf_f \dist(0,\phi_n(I_f))=1,
\end{equation}
where the infimum is over all faces $f\in F(T)\setminus\{o\}.$
Notice that
\begin{equation}\label{norm2}
\inf_f \dist(0,\phi_n(I_f))=\min_{f\in B_2} \dist(0,\phi_n(I_f))
\end{equation}
if $d_{gr}(o,\partial T)\geq2.$
Now assume that the 
unbiased disc triangulations $(T_n,o_n)$
converge in distribution to $(T,o)$.
We claim that the sequence $(T_n,o_n,\phi_n)$  has a subsequential
distributional limit $(T,o,\phi)$ with respect to compact convergence, and that $\phi$
is a normalized conformal map of $(T,o).$
To see this, we  first prove tightness of the law $\P_n$ of $(T_n,o_n,\phi_n)$.
As $(T_n,o_n) \to (T,o)$ in distribution, if $\epsilon >0$ is given, we can choose $L_r$ 
such that the set $A$ of all triples  $(T,o,\phi)$, with 
$$\max_{v\in B_r(o)} \deg{v} < L_r$$ 
for all $r$ 
and $\phi$ a normalized map of $(T,o),$
has
\begin{equation}\label{tight}
\P_n[A]\geq 1-\eps
\end{equation}
for each $n.$ To see that $A$ is compact, fix a (deterministic) sequence $(T_n,o_n,\phi_n)\in A$
and assume without loss of generality that $(T_n, o_n)$ is already converging 
(as there are a bounded number of possible $r$-neighborhoods of the root for each $r\geq1$,
every sequence of triangulations in $A$ has a  subsequence that converges with respect to $d_{gr}$).
Then fix $r$ and consider the sequence $\phi_n$ restricted to $B_r$. For $n$ large enough
$B_r(T_n)$ is isomorphic to $B_r(T)$, and we can assume that all $\phi_n$ are defined on the same 
subset $B_r$ of $T.$ Denote $c=c_{o}$ the center of the root triangle, and $c'$ the center of a fixed neighboring triangle of the root. Since 
the interstices have disjoint interiors,  it follows from \eqref{geometry} and \eqref{norm} that
$$\diam \phi_n(I_o) \leq C \inf_{f\neq o}  \dist(0,\phi_n(I_f)) =C.$$
On the other hand, \eqref{centergeom} implies
$$1\leq |\phi_n(c)-\phi_n(c')|\leq C\diam \phi_n(I_o)$$
and we obtain 
\begin{equation}\label{eq8}
1\leq |\phi_n(c')|\leq C^2.
\end{equation}
Since the family $\phi_n$ omits the values $0, \phi_n(c')$ and $\infty$ on $B_r\setminus\{o,o'\}$, 
Montel's theorem implies that it is a normal family. 
By a diagonal process we can extract a subsequence that converges compactly on every $B_r,$
thus on $R(T).$ It follows from \eqref{eq8} that the limit is non-constant, and by \eqref{norm2} it is a normalized
conformal map of $R(T)$. Thus $A$ is compact, and $\P_n$ is tight.
Now Prokhorov's theorem shows existence of subsequential distributional limits, where $\phi$ is a normalized conformal map of $(T,o).$ 
Along this subsequence,  $E_n = (T_n, o_n,g_n)$ satisfies the assumptions of 
Proposition \ref{parabolic-limit}, where $g_n$ denotes the restriction of $\phi_n$
to the centers $c_f$ of the faces $f$ of $T_n,$ renormalized to satisfy \eqref{cenorm}.
Indeed, if $T_n$ and $T'_n$ are equivalent as unrooted
triangulations, then $\phi_n=a \phi + b$ and $\phi'_n=a' \phi + b'$ by \eqref{phiT}
so that $g_n$ and $g'_n$ are similar.
It follows from Proposition \ref{parabolic-limit}, that the sequence of centers has only one 
accumulation point in $\C$. Hence it has at most two accumulation points in $\overline\C$. Since $d_{gr}(o,\partial T) = \infty$, 
Proposition \ref{parabolic} (more precisely Remark \ref{rmk}) implies that $R(T)$ is parabolic a.s.
\end{proof}

\section{Applications}

\subsection{Parabolicity of the UIPT}

Following Angel and Schramm, denote
$\mathcal{T}_n \subset \mathcal{T}$ the set of all rooted triangulations of $S^2$ with $n$ vertices,
 and  denote $\tau_n$ the uniform measure on $\mathcal{T}_n$. More precisely,  for $j=2$ and $j=3$ denote $\tau_n^{j}$ the uniform
measure on the the set of triangulations of type II  (no loops, but double edges allowed) and of type III (neither loops nor double edges), and
let  $\tau_n$ denote either of the two measures.  The following statements are proved in \cite{AS}:

\vspace{2mm}\noindent{\bf Theorem B.} (Theorem 1.8, Theorem 1.10, and Corollary 4.5 in \cite{AS}) {\em
The measures $\tau_n$ converge in distribution %as $n \to \infty$ 
to a measure $\tau$ supported on infinite triangulations in $\mathcal{T}$,
called the uniform infinite planar triangulation (UIPT).  
Samples from $\tau$ have almost surely one topological end.  Moreover, for each fixed $r$,
the law of the maximal vertex degree  in the ball of radius $r$ (centered at the root)
is tight.}
\vspace{2mm}

\begin{proof}[Proof of Corollary \ref{coruipt}]
By the Euler formula, a triangulation of the sphere with $n$ vertices has
$2n-4$ faces and $3n-6$ edges.
In order to apply Theorem \ref{main}, we need to approximate the UIPT by finite unbiased disc triangulations with $d(o_n,\partial T_n) \to \infty$.  To this end, let $S_n$ be an unrooted triangulation 
chosen according to the distribution induced by $\tau_n$ (which is not quite uniform due to triangulations with non-trivial symmetries).  
Choose two of the $2n-4$ faces $\Delta$ and $o$ uniformly at random. Then $(S_n, o)$ is a uniform rooted triangulation, and
$o$ is uniformly distributed over the triangles of $S_n\setminus\Delta.$
Now let $(T_n, o_n)$ be the random rooted disc triangulation obtained by removing $\Delta$ from $S_n$ and by setting $o_n=o.$ 
By tightness of the number of vertices in the ball $B_r$, one finds that
\[ d_{gr}(o_n,\partial T_n) = d_{gr}(o, \partial (S_n\setminus \Delta)) \to \infty \]
in distribution.
It follows that $(T_n, o_n)$ and $(S_n, o_n)$ have the same distributional limit, namely the UIPT. 
By Theorem B, the limit is one-ended almost surely, so that all assumptions of Theorem \ref{main} are satisfied and the UIPT is parabolic.

\end{proof}

\subsection{The infinite necklace}\label{intronecklace}

\subsubsection{Setup}

We begin with a description of the ``necklace''  construction of Random maps due to Scott Sheffield, 
and then show that the assumptions of Theorem \ref{main} are satisfied.

At eachtime $j$ of the inductive construction, we will have a triangulation of a disc $D_j$ in the closed upper half plane, consisting of $j$ triangles. Each vertex has one of two colors, say blue or red. The boundary of $D_j$ has one marked ``active'' edge with one blue endpoint $b_j$ and one red endpoint $r_j$, along which the growth of $D_j$ takes place. Initially ($j=0$) there are no triangles, we color the non-positive integers blue, the positive integers red, and set the active edge as the line segment $[0,1].$ To pass from $D_j$ to $D_{j+1}$, we glue a triangle to $D_j$ along the marked edge. There are four possible ways of glueing the new triangle, depicted by an explicit example in Figure \ref{fig:rbtriangle}.
We label these four choices as $B,b,R$ and $r$.

\begin{figure}[htp]
\centering
\includegraphics[scale=0.40]{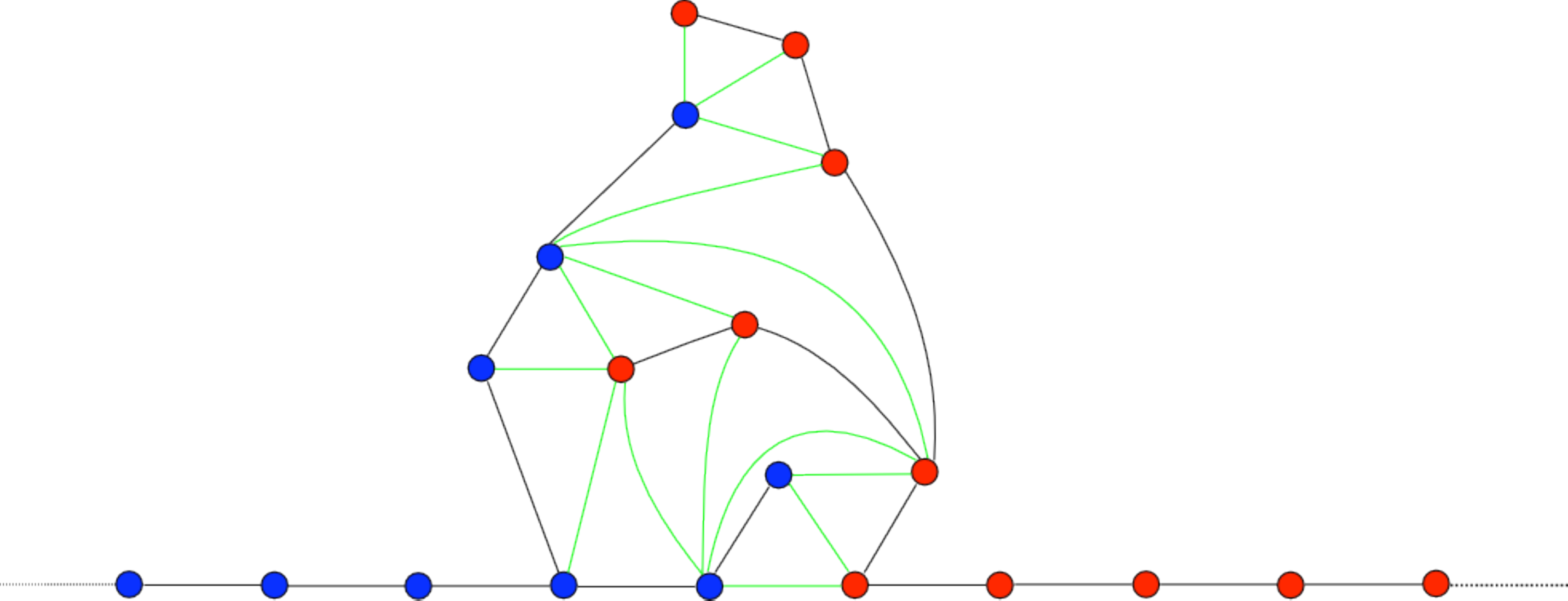}
\caption{The upper half plane necklace for the sequence $BRbRRbBBrrRBRR$}\label{fig:rbtriangle}
\end{figure}

%Adding a blue triangle 
Choice B consists of adding a new blue vertex $b_{j+1}\in\H\setminus D_j$,  adding two edges $(b_{j+1},b_j)$ and $(b_{j+1},r_j)$, and setting $r_{j+1}=r_j.$ Thus $(b_{j+1},r_j)$ is the new active edge, and $D_{j+1}$ is the union of $D_j$ with the triangle $t_j$ with vertices $b_j, r_j, b_{j+1}.$ Similarly, 
%adding a red triangle 
choice R consists of adding a red vertex $r_{j+1}$,  adding the edges $(r_{j+1},b_j)$ and $(r_{j+1},r_j)$, and setting $b_{j+1}=b_j.$ 
%Removing a blue triangle 
Choice b involves two slightly different cases: in the case that $b_j\in\H$ consists of letting $b_{j+1}$ be the counterclockwise neighbor of  $b_j$ in the boundary  of $D_j$, 
setting $r_{j+1}=r_j$, and adding the new active edge $b_{j+1},r_{j+1}$.
%Removing a blue triangle 
Choice b in the other case that  $b_j$ is an integer $m$,  we set $b_{j+1}=m-1$, add the edge $[m,m-1]$, again set $r_{j+1}=r_j$, and again add the new active edge $b_{j+1},r_{j+1}.$ In either case, a triangle with vertices $(b_{j+1}, b_j,r_{j+1})$ is added.
%Similarly, removing a red triangle means to 
Similarly, choice r sets $r_{j+1}$ as the clockwise neighbor of $r_j$ (and to set $r_{j+1}=m+1$ and  adds the edge $[m,m+1]$ if $r_j=m$ is a positive integer),  sets $b_{j+1} = b_j$, and adds the edge $(b_j, r_{j+1}).$

\begin{defn}

Given a finite word $X=x_1 x_2...x_n$ in the letters $x_i\in\{B,b,R,r\},$ we denote $T_{+}(X)$ the disc triangulation $D_n$ obtained by succesively applying the choices $x_i$. Similarly, we denote $T_{-}(X)$ the triangulation we obtain if we add triangles to the lower half plane instead of the upper half plane. Both $T_+$ and $T_-$ are rooted triangulations, whose root is the first triangle constructed. We denote $T(X,k)$ the rooted triangulation obtained from $T_+(X)$ by choosing the triangle created at $x_k$ as the root.

\end{defn}

%Adding a triangle to the upper half plane results in the same triangulation as removing a triangle (of the same color) from the lower half plane, and vice versa. 
Note the following simple fact, which is key to our approach: 
Write $B'=b, R'=r, b'=B, r'=R.$ If $X=x_1 x_2 ... x_n$ and $Y=y_1 y_2 ... y_m$, and if $Z=Z(X,Y)=y_m' y_{m-1}' ... y_1' x_1 x_2 ... x_n,$ then
\begin{equation}\label{fund}
T(Z, m+1) = T_+(X) \cup T_-(Y).
\end{equation}

Here we interpret the right hand side as a triangulation of the set $D_n(X)\cup D_m(Y)$ with root triangle the root of $D_n(X).$ We would like to extend the definition to infinite sequences $X=x_1 x_2...$  in a natural way. Notice however that the corresponding graph is not neccessarily even locally finite. We will show below that in the probabilistic setting, almost surely such pathologies do not occur.
If $X$ and $Y$ are two random infinite sequences with respect to the uniform measure (that is, the $x_i$ are i.i.d. with $P(x_i=x)=1/4$ for all $x\in\{B,b,R,r\}$), then we can think of $T_+(X)$ and $T_-(Y)$ as random triangulated halfplanes (``Quantum disks''), and interpret $T_+(X) \cup T_-(Y)$ as a glueing of two independent quantum disks. In the proof of Corollary \ref{corneck} we will show that 
$T_+(X) \cup T_-(Y)$ is indeed a disc triangulation. Provide $T_+(X) \cup T_-(Y)$ with a Riemann surface structure  by glueing equilateral triangles as in Section \ref{S:2.2}, and denote the resulting surface $R(X,Y)$. Our main result of this section is a proof of a conjecture of Scott Sheffield:

\begin{thm} The random Riemann surface $R(X,Y)$ is almost surely parabolic.
\end{thm}

In order to apply Theorem \ref{main}, we would like to construct a sequence of random finite disc triangulations converging to $T_+(X) \cup T_-(Y)$ in distribution. This will be achieved by $$T_n=T(X,k)$$ where $X=x_1...x_n$ is uniform and $k$ is uniformly chosen from $\{1,2,...,n\}$.

\begin{lemma} The law of $T_n$ converges in distribution to the law  of $T_+(X) \cup T_-(Y)$ with respect to $d_c$.
\label{rblimit}
\end{lemma}

The proof of the lemma is not difficult (see Section \ref{S:poin}), but  non-trivial: The discussion  of Example \ref{expl} showed that the analog
of Lemma \ref{rblimit} is false if  the infinite necklace is replaced by the 7-regular graph.

\subsubsection{Estimates}

By means of  \eqref{fund}, many questions about $T_n$ can be reduced to questions about $T_+(X)$. With each $X=x_1 x_2 x_3...$ we associate a function $k\mapsto S_k\in \mathbb{Z}^2$  as follows: $S_0 = (0,0), $ and $S_{k+1}=S_k + a$ where $a=(1,0), (-1,0), (0,1), (0,-1)$ according to $x_{k+1}=B,b,R,r.$  If $X$ is uniformly random, then $S_k$ is a simple random walk.  

Write $S_k=(X_k,Y_k).$   We will now examine properties of $T_+(X)$ by studying the corresponding properties of the simple random walk.  We start by exploring the distribution of the degrees of the vertices.  Suppose a blue vertex $v$ was created 
at time $k.$ At that time, it is connected to two edges, so that $\deg(v)=2.$
At a later time $j>k,$  an additional edge will connect to $v$ if and only if 
$X_i \geq X_k$ for all $k \leq i \leq j$, and either $X_j = X_k$ or $X_{j-1} = X_k$ (or both).
In other words, a new edge connects to $v$ for each time $j$ at which the walk $X_j$ lands on $X_k$ or goes from $X_k$ to $X_k+1$.
At the first time $j$ where the walk goes from $X_k$ to $X_k-1,$ the newly formed edge separates $v$  from the boundary of $D_j$
so that $v$ becomes an interior vertex and cannot be connected to any edge that is created in the future.  
A similar analysis applies to red vertices.  
In particular, the vertex $v=0$ is in the outer boundary (=the boundary minus $\mathbb Z$) 
of $T_+(x_1 x_2...x_k)$ if and only if
\begin{equation}\label{outerboundary}
X_j\geq0 \text{ for all } 1\leq j\leq k.
\end{equation}
Similarly, the degree of $v=0$ in $T_+(x_1 x_2...x_k)$ satisfies
\begin{equation}\label{degree0}
\deg(0) \leq 2 |\{0\leq j \leq k: b_j=0\}| = 2 |\{j \leq k: X_j=0, X_i\geq0 \text{ for all } i<j\}|.
\end{equation}

\noindent
We use this to show exponential decay of the law of the degrees of the vertices:

\begin{lemma}\label{degree}
For each $n,$ each $j\in\{1,2,...,n\}$ and each of the vertices  $v$ of the $j$-th face of $T_n$, we have 
\begin{equation}
\mathbb{P}[\deg(v)\geq m] \leq 2 \left(\frac34\right)^{m/4}.
\end{equation}
\end{lemma}

\begin{proof}
The degrees of the vertices of $T_n=T(X,k)$ are obviously independent of the choice $k$ of the root triangle.

By \eqref{fund}, we can write
$$T_n=T(x_1...x_n) = T_+(x_j...x_n) \cup T_-(x_{j-1}'...x_1'),$$
so that the vertices of the $j-th$ face of $T_n$ appear as the vertices of the root faces of
$T_+(x_j...x_n)$ and $T_-(x_{j-1}'...x_1').$
We assume that the vertex $v$ of  the face $f_j(T_n)$ corresponds to the vertex 0 
of $T_+(x_j...x_n)$ (the other two cases are similar). Then, if $d_+$ and $d_-$ denote the degrees
of $v$ in $T_+(x_j...x_n)$ and in $T_-(x_{j-1}'...x_1'),$ we have
$\deg(v)\leq d_+ + d_-.$
It follows that
$$\P_n[\deg(v)\geq m] \leq 2\P[d_+\geq m/2].$$ 
By \eqref{degree0}, we have
$$\P[d_+\geq k/2] \leq \P[ |\{j \leq m: X_j=0, X_i\geq0 \text{ for all } i<j\}|\geq m/4] \leq (\frac34)^{m/4}.$$

\end{proof}

\noindent
Next, we show that the boundary of $T_n$ is small compared to $n.$
\begin{lemma} \label{boundarysize}
For all $t>0$
\begin{equation}
\PP[|\partial T_n|\geq  t \sqrt n ] \leq   2\exp\left(-\frac{t^2}{32}\right).
\end{equation}
\end{lemma}

\begin{proof}
Since  the boundary of $T(X,k)$ does not depend on $k,$  the claim is equivalent to the claim
$$\PP[|\partial T_+(x_1...x_n)|\geq  t \sqrt n ] \leq  2\exp\left(-\frac{t^2}{32}\right).$$
A blue vertex $v=b_{k}$ is in the boundary if either $v$ is a negative integer $m$, in which case $X_k=m,$ or if $v$ is not an integer and $X_j\geq X_K$ for all $k<j\leq n.$ Writing $a=|\min_{1\leq k\leq n}  X_k|$ and $b=|\min_{1\leq k\leq n}  Y_k|$, it follows that
$$|\partial T_+(x_1...x_n)| \leq  a+b+|X_n+a| + |Y_n+b|.$$
Thus, by symmetry,
$$\PP[|\partial T_+(x_1...x_n)|\geq  t \sqrt n ] \leq \PP[a\geq  \frac t4 \sqrt n] \leq 
2\PP[X_n \leq - \frac t4\sqrt n].$$
By Azuma's inequality, 
%$$\PP[X_n \leq - \frac t4\sqrt n] \leq \exp\left(-\left(\frac t4\right)^2/2\right)$$
$$\PP[X_n \leq - \frac t4\sqrt n] \leq \exp(-(\frac t4)^2/2)$$
and the claim follows.
\end{proof}

Finally, we conclude that, with large probability, the root has large distance from the boundary:
\begin{lemma} \label{boundarydistance} Let $r$ be the root vertex.
There is a constant $C$ such that
\begin{equation}
\PP[d(r,\partial T_n)\leq  k ] \leq  \frac{ (C\log n)^{k+1}}{\sqrt{n}}.
\end{equation}
In particular, for each fixed $k,$ 
\begin{equation}
\PP[d(r,\partial T_n)\leq  k ] \to 0 \text{ as } n \to\infty.
\end{equation}
\end{lemma}

\begin{proof}
Let 
$$M_n = \max \{\deg(v): v \text{ vertex of } T_n \}.$$
By Lemma \ref{degree},
$$\PP[M_n\geq k ]\leq \sum_{v \in T_n} \PP[\deg(v)\geq k]\leq  n \left(\frac34\right)^{\frac{k}{4}} \leq\frac1n$$
if $k\geq C \log n$ with $C=8/ \log (4/3).$
By Lemma \ref{boundarysize} there is $C$ such that
$$\PP[|\partial T_n|> C\sqrt{n\log n}]<\frac1n.$$
On the event $\{M_n \leq C\log n; |\partial T_n|\leq C\sqrt{n\log n} \}$ we have 
$$|\{v: d(v,\partial T_n)\leq k\}| \leq |\partial T_n| (C\log n)^k\leq \sqrt{n}(C\log n)^{k+1}.$$
Since every triangle appears equally likely as the root triangle,
$$\PP[d(r,\partial T_n)\leq  k ] \leq \frac2n + \frac{(C\log n)^{k+1}}{\sqrt{n}}$$
and the lemma follows.
\end{proof}

\subsubsection{Parabolicity of the infinite necklace}\label{S:poin}

\begin{proof}[Proof of Corollary \ref{corneck}]

We wish to apply Theorem \ref{main} to   $T_+(X) \cup T_-(Y)$ and thus prove parabolicity of the Riemann surface $R(X,Y)$. 
To this end, we will realize $T_+(X) \cup T_-(Y)$ as the distributional limit of the random rooted finite unbiased disc triangulations
$T_n.$  Lemma 3.6 immediately implies that $d(o_n,\partial T_n) \to \infty$ in law.
It remains to show Lemma \ref{rblimit}, and one-endedness.

To show Lemma \ref{rblimit}, write $T_+(X) \cup T_-(Y)|_m$ for the disc triangulation  given by only using the first $m$ elements of $X$ and $Y$,  namely 
$$T_+(X) \cup T_-(Y)|_m=T_+(x_1,x_2,...,x_m) \cup T_-(y_1,y_2,...,y_m).$$
By the discussion before Lemma \ref{degree}, almost surely the ball of radius $r$ around the root is determined at a finite time,
\begin{equation}\label{eqbr}
B_r(T_+(X) \cup T_-(Y)|_m) = B_r( T_+(X) \cup T_-(Y)|_{m'})
\end{equation}
for all $m,m'$ large. This gives a rigorous definition of the random triangulation $T_+(X) \cup T_-(Y)$ and thus the  Riemann surface $R(X,Y)$.
Fix a possible rooted combinatorial ball $B$ of radius $r$. We need to show that   
$$\P[ B_r(T_n)=B]\to \P[ B_r(T_+(X) \cup T_-(Y))=B]$$ 
as $n\to\infty.$
It follows from \eqref{eqbr} that
$$ | \P [B_r (T_+(X) \cup T_-(Y) )=B] - \P[B_r (T_+(X) \cup T_-(Y)|_m)=B] | < \epsilon $$
for every $\epsilon$ and $m\geq m_{\epsilon}.$
With $T_n=T(x_1\cdots x_n,k),$ since $\P[m<k<n-m]=1-2m/n$  we therefore have by \eqref{fund}
$$ | \P [B_r (T_+(X) \cup T_-(Y) )=B] - \P[B_r (T_n)=B] | <\epsilon + \frac{2m}{n}$$
and the lemma follows.

To show that $T_+(X) \cup T_-(Y)$ is one ended,  fix a finite subgraph  $H$  and choose $m$ so large that
$H$ is contained in the interior of the finite disc triangulation $D=T_+(X) \cup T_-(Y)|_m.$ The complement of $D$ is connected 
(through $\partial D$), hence $T_+(X) \cup T_-(Y) \setminus H$ has only one infinite connected component.

Now Theorem 1.1  shows that R(X,Y) is almost surely parabolic.
\end{proof}

%\end{doublespace}

\end{document}